\documentclass[oneside,10pt,leqno]{amsart}
\usepackage{amssymb,amsmath,latexsym,amscd}
\usepackage{array,hhline}

\usepackage{longtable}

\newcommand{\diag}{\operatorname{diag}}

\def\ga{\mathfrak{a}}

\def\gg{\mathfrak{g}}
\def\gh{\mathfrak{h}}

\def\gk{\mathfrak{k}}

\def\gm{\mathfrak{m}}

\def\gp{\mathfrak{p}}

\def\gs{\mathfrak{s}}
\def\gt{\mathfrak{t}}
\def\gu{\mathfrak{u}}
\def\gv{\mathfrak{v}}
\def\gw{\mathfrak{w}}

\def\gz{\mathfrak{z}}

\def\C{\mathbb{C}}

\def\H{\mathbb{H}}

\def\R{\mathbb{R}}

\def\Z{\mathbb{Z}}

\def\bI{\mathbf{I}}

\def\Im{{\rm Im}\,}
\def\Re{{\rm Re}\,}

\def\Ad{{\rm Ad}}

\def\tr{{\rm trace\,}}

\def\diag{{\rm diag}}

\renewcommand{\thesection}{\arabic{section}}

\renewcommand{\thetable}{{\large \thesection.\arabic{equation}}}
\newtheorem{theorem}[equation]{Theorem}

\newtheorem{lemma}[equation]{Lemma}

\newtheorem{corollary}[equation]{Corollary}

\newtheorem{proposition}[equation]{Proposition}

\def\sideremark#1{\ifvmode\leavevmode\fi\vadjust{\vbox to0pt{\vss
 \hbox to 0pt{\hskip\hsize\hskip1em
\vbox{\hsize2cm\tiny\raggedright\pretolerance10000 
 \noindent #1\hfill}\hss}\vbox to8pt{\vfil}\vss}}} 

\begin{document}

\title{Local and Global Homogeneity for Three Obstinate Spheres}

\author{Joseph A. Wolf}\thanks{Research partially supported by a Simons
Foundation grant}
\address{Department of Mathematics \\ University of California, Berkeley \\
	CA 94720--3840, U.S.A.} \email{jawolf@math.berkeley.edu}

\date{file /texdata/working/obstinate.tex, last edited 18 May 2020}



\begin{abstract}
In this note we complete a study of globally homogeneous Riemannian 
quotients
$\Gamma\backslash (M,ds^2)$ in positive curvature.  Specifically, $M$ is a 
homogeneous space $G/H$ that admits a $G$--invariant Riemannian metric of 
strictly positive sectional curvature, and $ds^2$ is a $G$--invariant 
Riemannian metric on $M$, not necessarily normal and not necessarily 
positively curved.  The Homogeneity Conjecture is that 
$\Gamma\backslash (M,ds^2)$ is (globally) homogeneous if and only if 
$(M,ds^2)$ is homogeneous and every $\gamma \in \Gamma$ is of constant 
displacement on $(M,ds^2)$.  In an earlier paper we verified that 
conjecture for all homogeneous spaces that admit an invariant 
Riemannian metric of positive 
curvature --- with three exceptions, all odd dimensional spheres, which 
surprisingly did not yield to the earlier approaches.  Here we develop 
some methods that let us verify the Homogeneity Conjecture for those three 
obstinate spheres.  That completes verification of the Homogeneity
Conjecture in positive curvature.
\end{abstract}

\maketitle

\section{Introduction.}\label{sec1} 
\setcounter{equation}{0}
\bigskip

In this note we study homogeneous spaces $M = G/H$ that admit a 
$G$--invariant Riemannian metric of strictly positive curvature.  
Let $ds^2$ be a $G$--invariant Riemannian metric on $M$, not necessarily 
normal and not necessarily positively curved.  We
consider Riemannian quotient manifolds $\Gamma\backslash (M ,ds^2)$
and ask when such a manifold is globally homogeneous.

In \cite{W2020} we verified
a certain conjecture, concerning global
homogeneity for locally homogeneous Riemannian manifolds
$\Gamma \backslash (M,ds^2)$, when $M = G/H$ admits
an invariant Riemannian metric of positive sectional curvature ---
with three exceptions.  In this note we deal with those exceptions.

Let $(M,ds^2)$ be a connected simply connected Riemannian homogeneous
space.  Let $\pi: M \to M'$ be a Riemannian covering.  In other
words $\pi: M \to M'$ is a topological covering space that is a local
isometry.  Then the base of the covering must have form $M' = 
\Gamma \backslash M$
where $\Gamma$ is a discontinuous group of isometries of $M$ such that
only the identity element has a fixed point.  Clearly $M'$, with the 
induced Riemannian metric $ds'^2$ from $\pi: M \to M'$, is locally 
homogeneous.  We ask when $(M',ds'^2)$ is globally homogeneous.

If $M' = \Gamma \backslash M$ is homogeneous then \cite{W1960} every
element $\gamma \in \Gamma$ is of constant displacement
$\delta_\gamma(x) = dist(x,\gamma x)$ on $M$.  For the identity
component of the isometry group $\mathbf{I}(M',ds'^2)$ lifts to the normalizer 
$N_{\bI(M,ds^2)}(\Gamma)$ of $\Gamma$ in the isometry group $\bI(M,ds^2)$,
and $N_{\bI(M,ds^2)}(\Gamma)/\Gamma$
is a transitive group of isometries on $(M',ds'^2)$.  Since $\Gamma$ 
is discrete the identity component of that normalizer actually
centralizes $\Gamma$ in $\bI(M,ds^2)$, and this centralizer is transitive
on $M$.  If $x, y \in M$ and $\gamma \in \Gamma$ we write 
$y = g(x)$ with $g$ in the centralizer of $\Gamma$.  Compute $\delta_\gamma(y)
= dist(y,\gamma y) = dist(g x,\gamma g x) = dist(gx,g\gamma x) 
= dist(x,\gamma x) = \delta_\gamma(x)$.  That is the easy half of the

\noindent {\bf Homogeneity Conjecture.}
{\it Let $M$ be a connected, simply connected Riemannian homogeneous manifold
and $M \to \Gamma \backslash M$ a Riemannian covering.  Then 
$\Gamma \backslash M$ is homogeneous if and only if every $\gamma \in \Gamma$
is an isometry of constant displacement on $M$.}

Over the years there has been a lot of work on this conjecture,
implicitly beginning in the thesis of Georges Vincent 
\cite[\S 10.5]{V1947},  who noted that the linear transformations
$\diag\{R(\theta), \dots , R(\theta)\}, R(\theta) = \left ( \begin{smallmatrix}
\cos(\theta) & \sin(\theta) \\ -\sin(\theta) & \cos(\theta) \end{smallmatrix}
\right )$, are of constant displacement on the sphere $S^{2n-1}$.  

I extended this to a proof of the Homogeneity Conjecture, first for
spherical space forms \cite{W1961} and then for locally symmetric
Riemannian manifolds \cite{W1962}.  The proof 
used classification and case by case checking.  This was partially 
improved by Freudenthal \cite{F1963} and 
Ozols (\cite{O1969}, \cite{O1973}, \cite{O1974})
for the case where $\Gamma$ is contained in the identity component 
of $\bI(M,ds^2)$.  

The Homogeneity Conjecture is valid for locally
symmetric Finsler manifolds as well \cite{DW2013}.

A number of special cases of the Homogeneity Conjecture have
been verified.  Rather that make a long list I'll just note that many
of them are listed in \cite{W2018} and \cite{W2020}.

In Section \ref{sec2} we recall some facts about homogeneous spaces
that admit a Riemannian metric of strictly positive sectional curvature.
We then establish some basic tools that we need for the open cases.

In Section \ref{sec3} we settle the first open case, the 3-sphere 
as the group manifolds $SU(2)$ with left translations.  The technique
is to use the Maurer--Cartan forms on the group.

In Section \ref{sec4} we go to the next open case, $SU(m+1)/SU(m) 
= S^{2m+1}, m \geqq 1$.  This uses elementary matrix methods.

In Section \ref{sec5} we go to the open case $Sp(m+1)/Sp(m) = S^{4m+3}$
with the restriction $m \geqq 2$.  This restriction is needed for some
Weyl group considerations.  We use a mixture of Weyl group methods,
split fibrations (\cite{W2018} and \cite{W2020}), 
and computation with quaternionic matrices.

In Section \ref{sec6} we go to the last open case, $Sp(2)/Sp(1) = S^7$,
where we draw on methods from Section \ref{sec5} but take advantage of
the specific setting.

Finally, in Section \ref{sec7}, we summarize the results of Sections
\ref{sec2} through \ref{sec6} and describe how this completes the proof of
the Homogeneity Conjecture for homogeneous manifolds that admit 
an invariant metric of strictly positive sectional curvature.

Along the way we describe the isometries of constant displacement.

\bigskip
\section{The Classification for Positive Curvature.}\label{sec2}
\setcounter{equation}{0}
\bigskip

Here are the three homogeneous spaces considered in this note.
The numbering is retained from \cite[Table 2.1]{W2020}.
The spaces and the isometry
groups are listed in the first two columns of Table \ref{obstinate-table}
below.  The third column lists some fibrations that will be relevant to our
verification of the Homogeneity Conjecture for manifolds that admit 
an invariant
metric of strictly positive curvature.  See \cite[Section 4]{Z2007} for a 
description of exactly which invariant metrics have positive sectional 
curvature.

\addtocounter{equation}{1}
{\footnotesize
\begin{longtable}{|r|l|l|c|}
\caption*{\bf {\normalsize Table} \thetable \quad {\normalsize
        The Three Obstinate Spheres}}
\label{obstinate-table} \\
\hline
 & $M = G/H$ & $\bI(M,ds^2)$ & $G/H \to G/K$  \\ \hline
\hline
\endfirsthead
\multicolumn{4}{l}{{\normalsize \textit{Table \thetable\, continued from
        previous page $ \dots$}}} \\
\hline
 & $M = G/H$ & $\bI(M,ds^2)$ & $G/H \to G/K$ \\ \hline
\hline
\endhead
\hline \multicolumn{4}{r}{{\normalsize \textit{$\dots$ Table \thetable\,
        continued on next page}}} \\
\endfoot
\hline
\endlastfoot
\hline
{\rm 15} & $S^{2m+1} = SU(m+1)/SU(m)$ & $U(m+1)\rtimes\Z_2$ & 
		$S^{2m+1} \to P^m(\C)$
        \\ \hline
{\rm 16} & $S^{4m+3} = Sp(m+1)/Sp(m)$ & $Sp(m+1)\rtimes_{\Z_2} Sp(1)$ &
		$S^{4m+3} \to P^m(\H)$
        \\ \hline
{\rm 17} & $S^3 = SU(2)$ & $O(4)$ & $S^3 \to P^1(\C) = S^2$
        \\ \hline
\end{longtable}
}

\begin{theorem}\label{conj-non-normal}
Let $M = G/H$ be a connected, simply connected homogeneous space.  Suppose 
that $M$ has a $G$--invariant Riemannian metric of strictly positive curvature.
Let $ds^2$ be any $G$--invariant Riemannian metric on $M$, not necessarily the 
normal or the positively curved metric. 
Suppose further that $M = G/H$ is not
one of the entries {\rm (15), (16)} or {\rm (17)} of 
{\rm Table \ref{obstinate-table}}.  Then the Homogeneity Conjecture
is valid for $(M,ds^2)$.
\end{theorem}

This is the main result of \cite{W2020}.  The purpose of this paper is to 
extend it to the cases of Table \ref{obstinate-table} as well.  This will
make use of a few simple observations.

\begin{lemma} \label{conj}
Let $g$ and $\gamma$ be isometries of a Riemannian manifold
$(M,ds^2)$.  Suppose that $\gamma$ is of constant displacement $c$.  Then
$g^{-1}\gamma g$ is of the same constant displacement $c$.  
\end{lemma}

\begin{proof} The distance $\rho(y,\gamma y) = c$ for all $y \in M$.  Compute 
$\rho(x,g^{-1}\gamma g x) = \rho(gx, \gamma gx) = c$ for all
$x \in M$.
\end{proof}

We extend Lemma \ref{conj} to geodesics.  By horizontal projection in the
tangent bundle of $G$ we mean projection to the horizontal subspaces for
the Levi--Civit\` a connection of $ds^2$.  One easily picks this out on
the Lie algebra when the representation $\Ad_G|_H$ of $H$ on the tangent
space $\gg/\gh$ is disjoint (no common summand) from the adjoint 
representation of $\gh$.

\begin{lemma}\label{s-conj}
Let $\gamma$ be an isometry of constant displacement $c$ in a homogeneous 
Riemannian manifold $(M,ds^2)$, where $M = G/H$ with $G$ connected and 
$ds^2$  $G$--invariant.  Let $t \mapsto \sigma(t)$ be a minimizing 
geodesic from $x_0 = 1H \in M$ to $\gamma(x_0)$, parameterized proportional 
to arc length with $\sigma(0) = x_0$ and $\sigma(1) = \gamma(x_0)$\,.  
Let $\pi: G \to M$ be the projection and let $\widetilde{\sigma}$ denote 
the lift of $\sigma$ to a horizontal curve in $G$ with 
$\widetilde{\sigma}(0) = 1$.  Let $g \in G$  and $\beta(t) =
\pi([\Ad(g)\widetilde{\sigma}(t)]g)$.  Then $\beta$ is a 
minimizing geodesic from $g(x_0)$ to $g(\gamma(x_0))$ and the
horizontal component of $\Ad(g)[\widetilde{\sigma}'(0)]$
has the same length $c$ as $\widetilde{\sigma}'(0)$.
\end{lemma}

\begin{proof}
Since the vector field $\widetilde{\sigma}'$ is basic, so are all its left
$G$--translates, and thus $\beta(t) = 
\pi(\Ad(g)[\widetilde{\sigma}(t)]g) = 
\pi(g\widetilde{\sigma}(t))$ is a minimizing geodesic from 
$\beta(0) = gx_0$ to $\beta(1) = g\gamma x_0$.
In particular the square length $||\beta'(0)||^2 = ||\sigma'(0)||^2 = c^2$.
Thus the horizontal component of $\Ad(g)[\widetilde{\sigma}'(0)]$,
which is the horizontal lift of $\beta'(0)$,
has the same length $c$ as $\widetilde{\sigma}'(0)$.
\end{proof}

\bigskip
\section{$SU(2) = Sp(1) = S^3$.}\label{sec3}
\setcounter{equation}{0}
\bigskip

We consider $S^3$ as the group manifold with $SU(2) = Sp(1)$ acting
by left translations.  
Let $\{\omega_1\,, \omega_2\,, \omega_3\}$ denote the (left--invariant) 
Maurer--Cartan forms for the group $SU(2)$.   Then the constant curvature
metrics are the $ds^2 = a(\omega_1^2 + \omega_2^2 + \omega_3^2$), $a > 0$\,  
with isometry group $\bI^0(S^3,ds^2) = [SU(2) \times SU(2)]/\{(I,I),(-I,-I)\}$ 
acting by $(g,h): x \mapsto gxh^{-1}$ and $\bI(S^3,ds^2) = \bI^0(S^3,ds^2) \cup
s\bI^0(S^3,ds^2)$ with $\Ad(s)(g,h) = (h,g)$.
Up to $O(4)$--conjugacy, every left--invariant Riemannian metric on $SU(2)$
has form $\sum a_i\omega_i^2$ with each $a_i > 0$.  Thus, for verification
of the Homogeneity Conjecture there are only three cases, as follows

\begin{lemma}\label{3cases}
The left $SU(2)$--invariant metrics on $S^3$, and their isometry
groups , are equivalent, up to $SO(4)$--conjugacy, to one of these:
\begin{itemize}
  \item[(1)] $ds^2 = \omega_1^2 + \omega_2^2 + \omega_3^2$ with
 $\bI(S^2,ds^2)$ as described above,
  \item[(2)] $ds^2 = \omega_1^2 + \omega_2^2 + a\omega_3^2$\,,
 $0 < a \ne 1$, with $\bI(S^2,ds^2) = SU(2) \times U(1)$, and
  \item[(3)] $ds^2 = \sum a_i\omega_i^2$ with $\{a_1,a_2,a_3\}$
 distinct and $\bI(S^2,ds^2) = SU(2) \times \{1\}$.
\end{itemize}
\end{lemma}
Let $\Gamma \subset \bI(S^3,ds^2)$ be a finite group of constant displacement 
isometries of $(S^3,ds^2)$, and $\gamma \in \Gamma$.  From 
\cite[Lemma 4.2.2]{W1962}, $\gamma$ has form 
$\pm (g,h) \in [SU(2) \times SU(2)]/\{(1,1),(-1,-1)\}$.  In other
words,
\begin{equation}\label{split}
\Gamma \subset G \text{ where } 
	G = \bI(S^3,ds^2) \cap [SU(2) \times H]/[\pm (1,1)]
\end{equation}
for a subgroup $H \subset SU(2)$.  Note that right translations by 
elements of $H$ are isometries.

Let $\rho$ be the
distance function and $c = \rho(1,gh^{-1})$.  If $x \in S^3$ now
$c = \rho(x,gxh^{-1}) = \rho(1,x^{-1}gx\cdot h^{-1}) 
= \rho(h, x^{-1}gx)$.  So $c$ is the distance from
$h$ to any conjugate of $g$. 
A minimizing geodesic segment $\sigma$ from $h$ to $g$ meets
$\Ad(SU(2))g$ orthogonally.  It follows that $\sigma$ 
is tangent at $g$ to the $SU(2)$--centralizer of $g$.  

Consider a perturbation $\{\sigma_t\}$ of $\sigma$ as a minimizing
geodesic from $h$ to $g_t \in \Ad(SU(2))g$.
If $g \ne \pm 1$ then each $\sigma_t$ meets $\Ad(SU(2))g = \Ad(SU(2)g_t$
orthogonally.  Thus $\Ad(SU(2))g$ is half way to the cut locus of $h$, 
$\dim \Ad(SU(2))g = 2$, and each centralizer $Z_{SU(2)}(xgx^{-1})$ has
dimension $1$.  It follows that the image of each $\sigma_t$, which includes
both $h$ and $g_t$\,, centralizes $h$.  In other words $h$ commutes with
every conjugate of $g$.  Those conjugates generate $SU(2)$, so $h = \pm 1$.
We have proved:

\begin{proposition}\label{pmg} 
Let $\Gamma \subset \bI(S^3,ds^2)$ be a finite group of 
constant displacement isometries of $(S^3,ds^2)$.  If
$\gamma = \pm (g,h) \in \Gamma$ and $g \ne \pm 1$ then $h = \pm 1$.
\end{proposition}

Now every $\gamma \in \Gamma$ is contained either in $SU(2)\times \{\pm 1\}$
or in $\{\pm 1\}\times H$.  If $\gamma = \pm (g,1) \in \Gamma$
and $\gamma' = \pm (1,h') \in \Gamma$ with $g \ne \pm 1 \ne h'$ then
$\gamma \gamma' = \pm (g,h') \in \Gamma$ violates Proposition \ref{pmg}.  Thus,
using (\ref{split}),

\begin{corollary}\label{pmG}
Let $\Gamma \subset \bI(S^3,ds^2)$ be a finite group of
constant displacement isometries of $(S^3,ds^2)$.  Then either
$\Gamma \subset [SU(2) \times \{\pm 1\}]/[\pm (1,1)]$ or
$\Gamma \subset [\{\pm 1\} \times H]/[\pm (1,1)]$.
\end{corollary}

Now consider the two possibilities.  First, if 
$\Gamma \subset [\{\pm 1\} \times H]/[\pm (1,1)]$ then $SU(2)$, acting 
by left translations, centralizes $\Gamma$.  
Then $\Gamma \backslash S^3$ is homogeneous. Now consider the other case:
$\Gamma \subset [SU(2) \times \{\pm 1\}]/[\pm (1,1)]$ and $\Gamma$
has at least one element $\pm (g,1)$ with $g \ne \pm 1$.  

In Case (1) of Lemma \ref{3cases} the right translation group $H = SU(2)$
is transitive on $S^3$, so $\Gamma \backslash S^3$ is homogeneous, and
the Homogeneity Conjecture is valid.  This is a special case of
\cite[Corollary 4.5.3]{W1962}.  

Now we may assume $ds^2 = \omega_1^2
+ a_2\omega_2^2 + a_3 \omega_3^2$ with $a_2 \ne 1$ and $a_2, a_3 > 0$.
Let $\xi_i \in \gs\gu(2)$ denote tangent vectors to $S^3$ at the identity
such that $\omega_i(\xi_j) = 0$ for $i \ne j$ and each 
$t \mapsto \exp(t\xi_i)$ has period $2\pi$.  We have 
$\gamma = \pm (g,1) \in \Gamma$ with $g \ne \pm 1$.  Passing to an
$SU(2)$--conjugate we may assume $g = \exp(t_0\xi_1)$ with $0 < t_0 < \pi$.
Thus $\gamma$ has displacement $t_0$.  This uses Lemma \ref{conj}.  
But passing to another 
$SU(2)$--conjugate we may assume $g = \exp(t_0\xi_2)$ so $\gamma$ has
displacement $a_2t_0 \ne t_0$\,.  (These conjugations are specific to the 
group $SU(2) = Sp(1)$.)  This contradiction shows that, in
cases (2) and (3) of Lemma \ref{3cases}, $\Gamma$ does not contain
an element $\pm (g,1)$ with $g \ne \pm 1$.  We conclude:

\begin{theorem}\label{conj4s3}
Let $ds^2$ be a left $SU(2)$--invariant Riemannian metric on $S^3$.
Let $\Gamma$ be a finite group of isometries of constant displacement
on $(S^3,ds^2)$.  Then the centralizer of $\Gamma$ in $\bI(S^2,ds^2)$
is transitive on $S^3$, so the quotient Riemannian manifold 
$\Gamma \backslash (S^3,ds^2)$ is homogeneous.  In other words, the
Homogeneity Conjecture is valid for $(S^3,ds^2)$.
\end{theorem}

\section{$SU(m+1)/SU(m) = S^{2m+1}, \,\,m \geqq 2$}\label{sec4}
\setcounter{equation}{0}
\bigskip

Denote $G = SU(m+1)$, $K = U(m)$ and $H = SU(m)$, so
$S^{2m+1} = G/H \to G/K = P^m(\C)$ is a circle bundle.  The
fiber over $z_0 = 1K$ is the center $Z_K$ of $U(m)$.  $G/H$
has tangent space $\gv \oplus \gz_K$ where $\gv$ is the tangent space
$\C^m$ of $G/K$ and $\gz_K$ is the center of $\gk$;
$\gv$ and $\gz_K$ are the (two) irreducible summands of the
isotropy representation of $H$.  

\begin{proposition} \label{su2u}
Let $ds^2$ be Riemannian metric on $M = S^{2m+1}$ invariant under
$G = SU(m+1)$.  Then either the isometry group $\mathbf{I}(M,ds^2)$ is
the orthogonal group $O(2m+2)$, or $\mathbf{I}(M,ds^2) = [U(m+1)\cup\nu U(m+1)]$
where $\Ad(\nu)$ is complex conjugation on $U(m+1)$.
In the first case $(M,ds^2)$ is the constant curvature $(2m+1)$--sphere,
and in the second case $ds^2$ is given by {\rm (\ref{metric})} below.
\end{proposition}

\begin{proof}
The isotropy subgroup of $U(m+1) \cup\nu U(m+1)$ is $U(m) \cup\nu U(m)$,
where $U(m)$ consists of all $(m+1)\times (m+1)$ unitary matrices of the form
$\left ( \begin{smallmatrix} k & 0 \\ 0 & 1 \end{smallmatrix}
\right )$ and $\nu$ gives complex conjugation on $U(m+1)$.  The isotropy
representation is the usual action of
$U(m)$ on $\gv\cong \C^m$ together with complex conjugation from $\nu$,
and is trivial on $\gz_K$\,.  That preserves any
$\Ad(H)$--invariant real inner product on $\gv + \gz_K$\,.
Thus $[U(m+1)\cup \nu U(m+1)] \subset \mathbf{I}(M,ds^2) \subset O(2m+2)$.
As $U(m+1)$ is a maximal connected subgroup of $SO(2m+2)$ it follows that
either $[U(m+1)\cup \nu U(m+1)] = \mathbf{I}(M,ds^2)$ or
$\mathbf{I}(M,ds^2) = O(2m+2)$.
\end{proof}

If $\mathbf{I}(M,ds^2) = O(2m+2)$, so $(M,ds^2)$ is the constant
curvature $(2m-1)$--sphere, I proved the Homogeneity Conjecture
some time ago \cite{W1961}.  

We now assume that 
$\mathbf{I}(M,ds^2) = [U(m+1)\cup \nu U(m+1)]$ and view $S^{2m-1}$
as the coset space $U(m+1)/U(m)$.  It will be convenient to use the
notation $\tilde G = U(m+1)$, $\tilde H = U(m)$ and
$\tilde K = U(m) \times U(1)$.

In $(m+1)\times (m+1)$ complex matrices, $\tilde K$ consists of all
$\left ( \begin{smallmatrix} k & 0 \\ 0 & \ell \end{smallmatrix}
\right )$ with $k^{-1} = k^* \in U(m)$ and $\ell \in U(1)$. $\tilde H$ 
is the subgroup $\ell = 1$.
Use diagonal matrices for Cartan subalgebras $\tilde \gt$ of $\tilde \gg$ and 
$\tilde \gk$.  Then $\tilde \gt = \tilde \gt' + \tilde \gt''$ where
$\tilde \gt' \subset U(m)$ and $\tilde \gt'' = \gu(1)$.  Using
$\varepsilon_j(\diag\{a_1,\dots , a_{m+1}\})
= a_j$\,,  The simple roots of $\tilde \gg$ are the 
$\psi_i = \varepsilon_i - \varepsilon_{i+1}$.  Let $E_{i,j}$ denote the 
matrix with $1$ in row $i$ 
and column $j$ and $0$ elsewhere.  It spans the $\varepsilon_i -
\varepsilon_j$ root space when $i\ne j$, and $\tilde \gt$ consists of the
$\sum a_iE_{i,i,}$.  Now
$$
\gv = \sum_{j=1}^m \Bigl ( (E_{j,m+1} - E_{m+1,j})\R +
	\sqrt{-1}\,(E_{j,m+1} + E_{m+1,j})\R \Bigr )
$$ 
and $\tilde \gk$ has center
$\gz_{\tilde K} = \sqrt{-1}\,(E_{1,1} + \dots + E_{m,m})\R 
	+ \sqrt{-1}\, E_{m+1,m+1}\R\,$, and $\gk$ has center
$$
\gz_K = \sqrt{-1}\,\left ( (E_{1,1} + \dots + E_{m,m}) 
	-m \sqrt{-1}\, E_{m+1,m+1}\right )\R\,.
$$
The complex projective space 
$P^m(\C) = G/K$ is a symmetric space of rank $1$ so every element
of $\gv$ is $\Ad(K)$--conjugate to an element of 
$\ga = (E_{m,m+1} - E_{m+1,m})\R$.  Thus every element of $\gv + \gz_K$ is
$\Ad(H)$--conjugate to an element of $\ga + \gz_K$\,.  In other words, 
if $\eta \in \gv + \gz_K$ then we have
constants $a', a'' \in \R$ such that 
$$
\eta = a'(E_{m,m+1} - E_{m+1,m}) 
+ a'' \sqrt{-1}\, ((E_{1,1} + \dots + E_{m,m}) - mE_{m+1,m+1}).
$$

\begin{lemma}\label{id-comp}
If $\gamma \in U(m+1)$ has constant displacement $c$ 
on $S^{2m+1}$,  then $\gamma$ is central in $U(m+1)$.
\end{lemma}
\begin{proof}
Suppose that $\gamma$ is not central in $U(m+1)$.  Take
a minimizing geodesic $\{t \mapsto \sigma(t)\}$ from $x_0 = 1H$ to 
$\gamma x_0$\,, parameterized proportional to arc length, such that  
$\sigma(0) = x_0$ and $\sigma(1) = \gamma x_0$\,. Then
$\sigma'(t) = \tfrac{d}{dt} \sigma(t)$, has constant length $c$ for
$0 \leqq t \leqq 1$.  Let $\eta = \sigma'(0)$. Using Lemmas \ref{conj} 
and \ref{s-conj} we replace
$\gamma$ by a conjugate and assume $\eta \in \ga + \gz_K$\,,  say
$\eta = \eta' +  \eta''$.  Let $\kappa$ denote the negative
multiple of the Killing form such that $\kappa(\nu,\mu) = 
- \Re(\tr (\nu\overline{\mu}))$ on $(m+1)\times (m+1)$ matrices.
Using $\Ad(H)$--invariance,  the metric satisfies 
\begin{equation}\label{metric}
ds^2|_{\gv} = b'\kappa|_{\gv}\,,\,\, ds^2|_{\gz_K} = b''\kappa|_{\gz_K}\,,
\,\,\text{ and } \,\,ds^2(\gt',\gz_K) = 0
\end{equation}
for some $b', b'' > 0$.
The displacement satisfies
$c^2 = b'\kappa(\eta',\eta') + b''\kappa(\eta'',\eta'')$.

The normal Riemannian metric is the case $b' = b''$.  There
it is known, from \cite{W2018} and \cite[Proposition 3.2]{W2020}, that
$\gamma$ is central in $G$.  We now assume $b' \ne b''$ and argue more or 
less as in the paragraph leading to Proposition \ref{pmg}.  From the
discussion above, 
$$
||\eta'||^2 = 2b'|a'|^2 \text{ and } ||\eta''||^2 = b''|a''|^2 (m+m^2)
	\text{ so } ||\eta||^2 = 2b'|a'|^2 + b''|a''|^2 (m+m^2).
$$
The Weyl group of $G$ acts by all permutations of the $E_{j,j}$ so we have
$g \in G$ that exchanges $E_{1,1}$ with $E_{m+1,m+1}$ and leaves fixed
the other $E_{j,j}$.  Let $\zeta = \Ad(g)\eta$.  Then 
$$
\begin{aligned}
\zeta &= a'(E_{m,1} - E_{1,m}) + a'' \sqrt{-1}\, 
	((E_{m+1,m+1} + E_{2,2} + \dots + E_{m,m}) - mE_{1,1})\\
&= a'(E_{m,1} - E_{1,m}) + a'' \sqrt{-1}\,
	((E_{1,1} + E_{2,2} + \dots + E_{m,m}) -mE_{m+1,m+1}) \\
	& \phantom{X} - a'' \sqrt{-1}\,(m+1)(E_{1,1} - E_{m+1,m+1}).
\end{aligned}
$$
Split $E_{1,1} - E_{m+1,m+1} 
   = \tfrac{1}{m}[(E_{1,1} + \dots + E_{m,m})-mE_{m+1,m+1}]
   + [E_{1,1} - \tfrac{1}{m}(E_{1,1} + \dots + E_{m,m})]$.
It belongs to $\gk = \gh + \gz_K$\,. 
Combining two terms using $1-\tfrac{m+1}{m} = -\tfrac{1}{m}$,
\begin{alignat*}{2}
\zeta &= a'(E_{m,1} - E_{1,m}) - a'' \sqrt{-1}\,(m+1)
		[E_{1,1} - \tfrac{1}{m}(E_{1,1} + \dots + E_{m,m})]
	&\qquad &\gh \text{ component, } \\
&- a'' \sqrt{-1}\, \tfrac{1}{m}([(E_{1,1} + E_{2,2} + \dots + E_{m,m}) 
		- mE_{m+1,m+1}] 
	&\qquad &\gz_K \text{ component. } 
\end{alignat*}
Thus the horizontal component of $\zeta$ is its $\gz_K$--component,
and that has square length $b''|a''|^2\,(\tfrac{m^2 + m}{m^2})$.
Comparing this with $||\eta||^2 = 2b'|a'|^2 + b''|a''|^2 (m+m^2)$
we have $a' = 0$ and $\tfrac{m^2 + m}{m^2} = (m+m^2)$.  So
$m^2 = 1$, which contradicts $m \geqq 2$.  That in turn contradicts
our assumption that $\gamma$ is not central in $U(m+1)$,  completing
the proof of Lemma \ref{id-comp}.
\end{proof}

\begin{lemma}\label{if-nu}
Let $\Gamma \subset \mathbf{I}(M,ds^2)$ be a subgroup such that every
$\gamma \in \Gamma$ is an isometry of constant displacement.  If
$\gamma = \nu g \in \Gamma\cap \nu U(m+1)$ then $m+1$ is even,
$\gamma^2 = -I \in U(m+1)$, and $\Gamma$ is $SU(m+1)$--conjugate to the 
binary dihedral group whose centralizer in $U(m+1)$ is $Sp(\tfrac{m+1}{2})$.
\end{lemma}
\begin{proof}
Let $\gamma \in \Gamma$ with $\gamma = \nu g$ and $g \in U(m+1)$.  
Suppose that $\gamma \ne 1$
Let $g = g'z$ with $g' \in SU(m+1)$ and $z$ central (thus scalar) 
in $U(m+1)$.  The centralizer of $\nu$ in $U(m+1)$ is the orthogonal
group $O(m+1)$.  Let $B$ denote the maximal torus in $SO(m+1)$
consisting of all
$$
\diag\{R(\theta_1), \dots , R(\theta_{m+1})\} \text{ 1f $m+1$ is even},\,\,
\diag\{R(\theta_1), \dots , R(\theta_m),1\} \text{ 1f $m+1$ is odd,}
$$
where $R(\theta) = \left ( \begin{smallmatrix} \cos(\theta) & \sin(\theta) \\
-\sin(\theta) & \cos(\theta) \end{smallmatrix} \right )$.
Following de Siebenthal \cite{dS1956}, $\nu g'$ is
$\Ad(SU(m+1))$--conjugate to an element  $\nu b \in \nu B$.  Proposition
\ref{su2u} says that $\gamma^2 \in U(m+1)$ is a scalar matrix,
say $\gamma^2 = cI$.  But $\gamma^2 = (\nu b z)^2 = 
\nu b \nu^{-1}\cdot \nu z \nu^{-1} \cdot bz = \overline{b} z^{-1} bz = 
\overline{b} b = b^2$, so $b^2 = cI$.  Define $\theta_i$ mod $2\pi$ by
$b = \diag\{R(\theta_1), \dots , R(\theta_{m+1})\}$.
Then $b^2 = 
\diag\{R(2\theta_1), \dots , R(2\theta_{m+1})\}$.  Since $b^2 = cI$ either
$c = +1$ and each $\theta_i = \pm \pi$ mod $2\pi$, or
$c = -1$ and each $\theta_i = \pm \pi/2$ mod $2\pi$.

If $c = 1$ then $\gamma^2 = 1$.  But $\gamma$ is also an isometry for
the constant curvature metric on $S^{2m+1}$, and there the only
fixed point free isometry of square $1$ is the antipodal map $-I$.
But $-I \notin \nu U(m+1)$ because it is central in 
$U(m+1)\cup \nu U(m+1)$; so we cannot have $c = 1$.  Thus $c = -1$,
in other words $\gamma^2 = b^2 = -I$.  In particular $m+1$ is even. 

Let $\Gamma_0 = \Gamma \cap U(m+1)$.  Then $\Gamma = \Gamma_0 \cup
\gamma\Gamma_0$\,, and the elements of $\Gamma_0$ are scalar matrices.
If $\gamma_0 \in \Gamma_0$ now $\gamma_0 = \gamma \gamma_0 \gamma^{-1}
= \ne g \gamma_0 g^{-1}\nu^{-1} = \nu \gamma_0 \nu^{-1} = \overline{\gamma_0}
= \gamma_0^{-1}$.  Thus $\Gamma$ is the binary dihedral group whose 
$U(m+1)$--centralizer is $Sp(\tfrac{m+1}{2})$.
\end{proof}

Combining Lemmas \ref{id-comp} and \ref{if-nu} with Proposition
\ref{su2u} we have

\begin{theorem}\label{conj4su}
Let $ds^2$ be an $SU(m+1)$--invariant Riemannian metric on $S^{2m+1}, 
m \geqq 2$.  Let $\Gamma$ be a finite group of isometries of constant
displacement on $S^{2m+1}$.  Then the centralizer of $\Gamma$ in
$\mathbf{I}(S^{2m+1}, ds^2)$ is transitive on $S^{2m+1}$, so the
Riemannian quotient manifold $\Gamma\backslash (SU(m+1), ds^2)$ is
homogeneous.  In other words, the Homogeneity Conjecture is valid for
$(S^{2m+1}, ds^2)$.
\end{theorem}

\medskip
\section{{\Large {\bf $Sp(m+1)/Sp(m) = S^{4m+3}$\,, $m \geqq 2$.}}}
\label{sec5}
\setcounter{equation}{0}
\bigskip

The first step here is to prove Proposition \ref{sp2o},
which is the analog of Proposition \ref{su2u}.

\begin{lemma}\label{trans}
Let $D$ be a compact connected Lie group acting transitively and
effectively on $S^{4m+3}\,, m \geqq 1$.  Suppose that $Sp(m+1)
\subset D$ but $SU(2m+2) \not\subset D$.  Then $D$ is one of
$Sp(m+1)\cdot Sp(1)$, $Sp(m+1)\cdot U(1)$, or $Sp(m+1)$. 
\end{lemma}
\begin{proof}
It is now classical from \cite{MS1943} and \cite{B1949} that the compact
connected Lie groups acting transitively on spheres are the linear groups
$SO(d)$,\, $U(d/2)$,\, $SU(d/2)$,\, $Sp(d/4)$,\, $Sp(d/4)\cdot U(1)$  and  
$Sp(d/4)\cdot Sp(1)$ on $S^{d-1}$;\, $G_2$ on $S^6$,\, $Spin(7)$ on $S^7$, 
and $Spin(9)$ on $S^{15}$.  So the cases of $S^{4m+3}$ are 
$SO(4m+4)$, $U(2m+2)$, $SU(2m+2)$, $Sp(m+1)\cdot Sp(1)$, $Sp(m+1)\cdot U(1)$,
$Sp(m+1)$ all for $m \geqq 1$, $Spin(7)$ for $m=1$ and $Spin(9)$ for
$m = 3$.  The restrictions $Sp(m+1) \subset D$ and
$SU(2m+2) \not\subset D$ eliminate $SO(4m+4)$, $U(2m+2)$, $SU(2m+2)$,
$Spin(7)$ and $Spin(9)$.  For the latter two, $\dim Spin(7) = 21 = \dim Sp(3)$,
so $Sp(3) \subset D = Spin(7)$ would imply $Sp(3) = Spin(7)$, which is false;
and similarly $\dim Spin(9) = 36 = \dim Sp(4)$, so
$Sp(4) \subset D = Spin(9)$ would imply $Sp(4) = Spin(9)$, which is false.
The lemma follows.
\end{proof}

Denote $G = Sp(m+1)$, $K = Sp(m)\times Sp(1)$ and $H = Sp(m) \subset K$,
where $m \geqq 1$.
Then $G/H = S^{4m+3}$and we have the projection $S^{4m+3} \to P^m(\H)
= G/K$.  $G/H$ has tangent space $\gv + \gw$ where $\gv$ is the
tangent space $\H^m$ of $G/K$ and $\gw$ is the tangent space $\Im \H$
of the fiber of $S^{4m+3} \to P^m(\H)$.  The isotropy representation
of $H$ is the natural representation of $Sp(m)$ on $\H^m = \gv$, and
on $\gw$ it is three copies of the trivial representation.

Let $\kappa' = \kappa|_\gv$ and $\kappa'' = \kappa|_\gw$ where 
$\kappa(\mu,\nu) = -\Re\tr(\mu\,\overline{\nu})$ with trace taken in $Sp(m+1)$.
Let $\{e_1,e_2,e_3\}$ be a $\kappa''$--orthonormal basis of $\gw$
and split $\kappa'' = \kappa_1 + \kappa_2 + \kappa_3$ accordingly.
Then
\begin{equation}\label{sp-metric}
ds^2|_\gv = b_0\kappa',\, ds^2|_\gw = b_1\kappa_1 + b_2\kappa_2 + 
	b_3\kappa_3\,,\, ds^2(\gv,\gw) = 0 \text{ and }
	ds^2(e_i,e_j) = 0 \text{ for } i\ne j
\end{equation}
for some positive numbers $b_0, b_1, b_2 \text{ and } b_3$\,.
The normal Riemannian metric on $G/H$ is the case where
$b_0 = b_1 = b_2 = b_3$ for a certain $b_0 > 0$.  We now normalize
$ds^2$ and assume $b_0 = 1$; this has no effect on $\mathbf{I}(M,ds^2)$
nor on which isometries have constant displacement.

\begin{proposition} \label{sp2o}
Let $ds^2$ be an $Sp(m+1)$--invariant Riemannian metric on $M = S^{4m+3}$,
$m \geqq 1$.
Then either $ds^2$ is invariant under $SU(2m+2)$, or 
$\mathbf{I}(M,ds^2) = Sp(m+1)\cdot L = (Sp(m+1)\times L)/\{\pm (I_{m+1},I_3)\}$
where $L$ is one of the following.

{\rm (1)} $L = Sp(1)$ acting on $Sp(m+1)/Sp(m)$ on the right.  $L$ acts
on the tangent space as multiplication by quaternion unit scalars
on $\gv$ and the adjoint representation of $Sp(1)$ on $\gw$.  This is
the case $b_1 = b_2 = b_3$\,.

{\rm (2)} $L = O(2) \times \Z_2$ acting on $Sp(m+1)/Sp(m)$ on the right.  
$L$ acts the tangent space as multiplication by an $O(2)\times \Z_2$ 
{\rm (}essentially circle{\rm )} group of 
quaternion unit scalars on $\gv$, $O(2)$--rotation on the $(e_1,e_2)$--plane 
in $\gw$, and the $\Z_2$--action $e_3 \mapsto \pm e_3$ on $\gw$\,.  
This is the case where two, but not all three, of the $b_i$ are equal,
for example where $b_1 = b_2 \ne b_3$\,.

{\rm (3)} $L = \Z_2^3$ acting on $Sp(m+1)/Sp(m)$ on the right.  $L$
acts on the tangent space by $\pm 1$ on $\gv$ and the $e_i \mapsto \pm e_i$
on $\gw$.  This is the case where $b_1$\,, $b_2$ and $b_3$ are all different.
\end{proposition}

\begin{proof}
Suppose that $ds^2$ is not invariant under $SU(2m+2)$.  Then 
Lemma \ref{trans} shows that the identity component 
$\mathbf{I}^0(M,ds^2)$ of the isometry group must be 
$Sp(m+1)\cdot Sp(1)$, $Sp(m+1)\cdot U(1)$, or $Sp(m+1)$.
Thus $\mathbf{I}(M,ds^2) = Sp(m+1)\cdot L$ where $Sp(m) \cdot L$ is the
normalizer of the isotropy subgroup $Sp(m)$ of $\mathbf{I}^0(M,ds^2)$
in the orthogonal group of $ds^2|_{\gv + \gw}$ and also preserves $\kappa''$.

First consider the case
$\mathbf{I}^0(M,ds^2) = Sp(m+1)\cdot Sp(1)$.  That is the 
case $b_1 = b_2 = b_3$\,.  Then
$S^{4m+3} = [Sp(m+1)\cdot Sp(1)]/[Sp(m)\times Sp(1)]$ where
$Sp(m)$ acts on the summand $\gv$ of the tangent space, and $Sp(1)$
acts on $\gv$ as quaternion unit scalars and on $\gw$ by its adjoint
representation..  

Second consider the case 
$\mathbf{I}^0(M,ds^2) = Sp(m+1)\cdot SO(2)$.
This is the case where two of the $b_i$ are equal but different from
the third.  By conjugacy we may suppose $b_1 = b_2 \ne b_3$ here.  
Then $S^{4m+3} = [Sp(m+1)\cdot SO(2)]/[Sp(m)\times SO(2)]$, and
$L = O(2) \times \Z_2$\,, so
$\mathbf{I}(M,ds^2) = Sp(m+1) \cdot (O(2) \times \Z_2)$,
as asserted.

Third consider the case $\mathbf{I}^0(M,ds^2) = Sp(m+1)$.
This is the case where the $b_i$ are all different.
Then $S^{4m+3} = Sp(m+1)/Sp(m)$, $L = \Z_2^3$\,, and
$\mathbf{I}(M,ds^2) = Sp(m+1) \cdot \Z_2^3$\,,
as asserted.
\end{proof}

We write the quaternion algebra $\H$ as $\R + {\mathbf i}\R 
+ {\mathbf j}\R + {\mathbf k}\R$.
In $(m+1)\times (m+1)$ quaternion matrices, $K$ consists of all
$\left ( \begin{smallmatrix} k & 0 \\ 0 & \ell \end{smallmatrix}
\right )$ with $k\in Sp(m)$ and $\ell \in Sp(1)$. $H$
is the subgroup $\ell = 1$.  Use diagonal matrices with entries 
in ${\mathbf i}\R$ for Cartan subalgebras $\gt$ of $\gg$ and
$\gk$.  Then $\gt = \gt' + \gt''$ where
$\gt' \subset Sp(m)$ and $\gt'' = \gu(1) \subset \gs\gp(1)$.  Using
$\varepsilon_j(\diag\{a_1,\dots , a_{m+1}\})
= a_j$\,,  The simple roots of $\gg$ are the
$\psi_i = \varepsilon_i - \varepsilon_{i+1}\,, i \leqq m, \text{ and }
\psi_{m+1} = 2\varepsilon_{m+1}\,$.  Let $E_{i,j}$ denote the
matrix with $1$ in row $i$ and column $j$ and $0$ elsewhere, as usual,
so $\gt$ consists of the $\sum a_iE_{i,i}$ with each $a_i \in {\mathbf i}\R$. 
Now 
$$
\gv = \sum_{j=1}^m \Bigl ( (E_{j,m+1} - E_{m+1,j})\R +
        ({\mathbf i}\R + {\mathbf j}\R + {\mathbf k}\R)
	(E_{j,m+1} + E_{m+1,j}) \Bigr )
\text{ and }
\gw = ({\mathbf i}\R + {\mathbf j}\R + {\mathbf k}\R) E_{m+1,m+1}.
$$  
The quaternion projective space
$P^m(\H) = G/K$ is a symmetric space of rank $1$ so every element
of $\gv$ is $\Ad(K)$--conjugate to an element of
$\ga = (E_{m,m+1} - E_{m+1,m})\R$.  Thus every element of $\gv + \gw$ is
$\Ad(H)$--conjugate to an element of $\ga + \gw$\,.  In other words,
if $\tilde\eta \in \gv + \gw$ then we have 
constants $a', a_\ell \in \R$ such that $\tilde\eta$ is $\Ad(H)$--conjugate to
\begin{equation}\label{pieces}
\eta = \eta' + \eta'' \text{ with } \eta' 
	= a'(E_{m,m+1} - E_{m+1,m}) \text{ and } 
\eta'' = (a_1{\mathbf i} + a_2{\mathbf j} + a_3{\mathbf k})E_{m+1,m+1}
	= \eta_1 + \eta_2 + \eta_3\,.
\end{equation}

\begin{lemma}\label{sp-comp}
If $\gamma \in Sp(m+1)$ has constant displacement $c > 0$
on $S^{4m+3}$, $m \geqq 2$, then $\gamma$ belongs to the centralizer
of $Sp(m+1)$ in $\mathbf{I}(S^{4m+3},ds^2)$.
\end{lemma}
\begin{proof}
Suppose that $\gamma$ does not centralize $Sp(m+1)$.  Take
a minimizing geodesic $\{t \mapsto \sigma(t)\}$ from $x_0 = 1H$ to
$\gamma x_0$\,, parameterized proportional to arc length, such that
$\sigma(0) = x_0$ and $\sigma(1) = \gamma x_0$\,. Then
$\sigma'(t) = \tfrac{d}{dt} \sigma(t)$, has constant length $c$ for
$0 \leqq t \leqq 1$.  Let $\eta = \sigma'(0)$. Using Lemmas \ref{conj}
and \ref{s-conj} we replace
$\gamma$ by a conjugate and assume $\eta = \eta' +  \eta''$ as in 
(\ref{pieces}).  Recall the expression (\ref{sp-metric}) for $ds^2$.
Then the displacement satisfies
\begin{equation}\label{sp-disp}
c^2 = b'\kappa(\eta',\eta') + b_1\kappa(\eta_1,\eta_1)
	+ b_2\kappa(\eta_2,\eta_2) + b_3\kappa(\eta_3,\eta_3)
= 2b'a'^2 + b_1a_1^2 + b_2a_2^2 + b_3a_3^2\,.
\end{equation}

The Weyl group of $G$ contains all permutations of the ${\mathbf i}E_{j,j}$\,, 
including a conjugation by $g \in Sp(m+1)$ that exchanges
${\mathbf i}E_{1,1}$ with ${\mathbf i}E_{m+1,m+1}$  and leaves fixed the 
other ${\mathbf i}E_{i,i}$.  Let $\zeta = \Ad(g)\eta$, suppose $m > 1$,
and compute
$
\zeta = a'(E_{m,1} - E_{1,m})
	+ (a_1{\mathbf i} + a_2{\mathbf j} + a_3{\mathbf k})E_{1,1}\,.
$
Thus $m > 1$ implies $a'(E_{m,1} - E_{1,m}) \in \gh$ and
$(a_1{\mathbf i} + a_2{\mathbf j} + a_3{\mathbf k})E_{1,1} \in \gh$ as well,
so $c = 0$ and $\gamma = 1$.
\end{proof}

\begin{lemma}\label{spm-comp}
Let $\Gamma \subset {\mathbf I}(M,ds^2)$ be a subgroup such that every
$\gamma \in \Gamma$ is an isometry of constant displacement.  Suppose
$m \geqq 2$ and that $ds^2$ is not $SU(2m+2)$--invariant.  Then $\Gamma$
centralizes $Sp(m+1)$ in ${\mathbf I}(M,ds^2)$.
\end{lemma}

\begin{proof}
First consider the case $L = Sp(1)$.  There $L$ is irreducible on the
subspace $\gw$ of the tangent space of $M = G/H$, so $b_1 = b_2 = b_3$
in (\ref{sp-disp}).  Thus the fiber of the projection
$M = Sp(m+1)/Sp(m) \to Sp(m+1)/[Sp(m)\times Sp(1)] = P^m(\H)$ is
the constant curvature $S^3$.  Now
the conditions (5.1) of \cite{W2020} 
valid for $M = Sp(m+1)/Sp(m) \to Sp(m+1)/[Sp(m)\times Sp(1)] = P^m(\H)$.
We quote them as follows.
\begin{equation}\label{new-setup}
\begin{aligned}
&G \text{ is a compact connected simply connected Lie group,}\\
&H \subset K \text{ are closed connected subgroups of $G$ such that } \\
& \phantom{XXXX}\text{(i) $M = G/H$\,, $M' = G/K$, and $F=H\backslash K$\,,} \\
& \phantom{XXXX}\text{(ii) $\pi: M \to M'$ is given by  $\pi(gH) = gK$\,,
                right action of $K$ on $G/H$\,, } \\
& \phantom{XXXX}\text{(iii) $M'$ and $F$ are Riemannian symmetric spaces, and}\\
& \phantom{XXXX}\text{(iv) the tangent spaces $\gm'$ for $M'$, $\gm''$
        for $F$ and $(\gm' + \gm'')$ for $M$ satisfy $\gm' \perp \gm''$}\,.
\end{aligned}
\end{equation}
The arguments of \cite[Lemmas 5.2 and 5.3]{W2020} go through 
without change to prove that $\Gamma$ centralizes $Sp(m+1)$.

Second consider the case $L = O(2)\times \Z_2$\,.  
In view of Lemma \ref{sp-comp}, either $\Gamma \cap Sp(m+1) = \{I\}$
or $\Gamma \cap Sp(m+1) = \{\pm I\}$, and we need only consider the
situation where $\Gamma \not\subset Sp(m+1)$.  

To start, let
$\gamma = (g,r,t) \in \Gamma$ where $g \in Sp(m+1)$, $r \in O(2)$ acting
on the subspace of $\gw$ spanned by $\mathbf{i}\R + \mathbf{j}\R$, 
and $t \in O(1) = \{\pm 1\}$ acting on $\mathbf{k}\R$.

Suppose first that $\det (r) = -1$.  Then $r$ is conjugate to 
$\diag\{1,-1\}$, so $\gamma^2 = (g^2,I_3) \in Sp(m+1)$, and 
$g^2 = \pm I_{m+1}$\,.
If $t = +1$ the square of the displacement is of the form 
$c_\gamma^2 = c_0^2 + b_2a^2$, and if $t = -1$ it is of the form 
$c_\gamma^2 = c_0^2 + b_2a^2 + b_3a^2$ where $a > 0$, $ds^2$ is
$b_1\kappa$ on $\mathbf{i}\R$, $b_2\kappa$ on $\mathbf{j}\R$, 
and $b_3\kappa$ on $\mathbf{k}\R$,
We can permute $\mathbf{i}\R$, $\mathbf{j}\R$ and $\mathbf{k}\R$ through
conjugation by unit quaternions, where $\gw = \Im\H$.  Then all three
$b_i$ are equal and we are in the setting of $L = Sp(1)$.
That contradiction forces $\det (r) = +1$.

We now have $\gamma = (g,r,t) \in \Gamma$ where $g \in Sp(m+1)$, 
$r = \left ( \begin{smallmatrix} \cos(\theta) & \sin{\theta}\\
-\sin(\theta) & \cos(\theta) \end{smallmatrix} \right )$ on the
plane $\mathbf{i}\R + \mathbf{j}\R$, and 
$t = \pm 1$ on $\mathbf{k}\R$.  If $t = +1$ the square displacement
$c_\gamma^2 = c_0^2 + b_1 a^2 + b_2a^2$, and if $t = -1$ it is
$c_0^2 + b_1 a^2 + b_2a^2 + b_3 a^2$\,.
Conjugation by an appropriate element of 
$\exp(\gv)$ exchanges $\mathbf{i}\R$ with an element of $\gs\gp(m)$,
sending $\gamma$ to an isometry $\gamma'$ of square displacement
$c_{\gamma'}^2 = c_0^2 + b_2a^2$ if $t = +1$, $c_0^2 + b_2a^2 + b_3a^2$
if $t = -1$.  But $c_{\gamma'}^2 = c_\gamma^2$, so $b_1a^2 = 0$ for
either $t = \pm 1$.  That contradicts $\det (r) = +1$.

We have shown that if $L = O(2)\times \Z_2$ then every
$\gamma \in \Gamma$ belongs to $Sp(m+1)$; thus $\Gamma$ centralizes $Sp(m+1)$.

Third suppose that $L = \Z_2^3 = \left ( \diag\{\pm 1,\pm 1, \pm 1\}
\right ) $.  Let $\gamma = (g,\diag\{t_1,t_2,t_3\}) \in \Gamma$.
If $t_1 = t_2 = t_3$ then the action of $\diag\{t_1,t_2,t_3\}$ on $\gw$ 
comes from $\pm I_{m+1} \in Sp(m+1)$, and
we may view $\gamma$ as an element of $Sp(m+1)$.  Now we may assume
$t_1 = t_2 = 1\,, t_3 = -1$.  Thus the square displacement of $\gamma$
is $c_\gamma^2 = c_0^2 + b_3a^2$.  Conjugation by $\mathbf{j}$ exchanges 
$\mathbf{i}\R$ and $\mathbf{k}\R$ and sends $\gamma$ to $\gamma'$
with square displacement $c_{\gamma'}^2 = c_0^2 + b_1a^2$\,, so
$b_1 = b_3$\,.  Thus $L \ne \Z_2^3$\,.  This contradiction completes the proof. 
\end{proof}

Combining Lemmas \ref{trans}, \ref{sp-comp} and \ref{spm-comp} with 
Proposition \ref{sp2o} and Theorem \ref{conj4su} we have

\begin{theorem}\label{conj4sp}
Let $ds^2$ be an $Sp(m+1)$--invariant Riemannian metric on $S^{4m+3}$,
$m \geqq 2$.  Let $\Gamma$ be a finite group of isometries of constant
displacement on $S^{4m+3}$.  Then the centralizer of $\Gamma$ in
$\mathbf{I}(S^{4m+3},ds^2)$ is transitive on $S^{4m+3}$.  In other words
the Homogeneity Conjecture is valid for $(S^{4m+3},ds^2)$.
\end{theorem}

\medskip
\section{{\Large {\bf $Sp(2)/Sp(1) = S^{7}$.}}}
\label{sec6}
\setcounter{equation}{0}
\medskip

The homogeneous space
$Sp(2)/Sp(1) = S^{7}$ is the case $m=1$ of $Sp(m+1)/Sp(m) = S^{4m+3}$.
Lemma \ref{trans} and Proposition \ref{sp2o} apply here, but Lemmas
\ref{sp-comp} and \ref{spm-comp} do not.  However we can salvage something
from their proofs.  
\begin{proposition}\label{sp2-comp}
Suppose that $ds^2$ is $Sp(2)$--invariant but not $SU(4)$--invariant.
If $\gamma \in Sp(2)$ has constant displacement $c > 0$
on $S^7$, then $\gamma \in L = Sp(1)$, so belongs to the centralizer
of $Sp(2)$ in $\mathbf{I}(S^7,ds^2)$.  In particular if $\Gamma \subset Sp(2)$ 
is a finite group of constant displacement isometries of $(S^7,ds^2)$
then $(S^7,ds^2)$ is homogeneous.
\end{proposition}
\begin{proof}
The proof is similar to the arguments of Section \ref{sec5} -- up to a point.
We use the Cartan subalgebra $\gt = \gt' + \gt''$ of $\gs\gp(2)$
where $\gt' = \mathbf{i}\R E_{1,1}$ and $\gt'' = \mathbf{i}\R E_{2,2}$.
Take 
a minimizing geodesic $\{t \mapsto \sigma(t)\}$ from $x_0 = 1H$ to
$\gamma x_0$\,, parameterized proportional to arc length, such that
$\sigma(0) = x_0$ and $\sigma(1) = \gamma x_0$\,. Then
$\sigma'(t) = \tfrac{d}{dt} \sigma(t)$, has constant length $c$ for
$0 \leqq t \leqq 1$.  Let $\eta = \sigma'(0)$. Using Lemmas \ref{conj}
and \ref{s-conj} we replace
$\gamma$ by a conjugate and assume $\eta = \eta' +  \eta''$ 
with $\eta' = a'(E_{1,2} - E_{2,1})$ and $\eta'' = a''\mathbf{i}E_{2,2}$\,,
$a'$ and $a''$ real.  Using (\ref{sp-metric}) the displacement satisfies
$c^2 = b'\kappa(\eta',\eta') + b_1\kappa(\eta_1,\eta_1)
	= 2b'a'^2 + b_1a''^2\,.$
Conjugation by $J = 
\left ( \begin{smallmatrix} 0 & 1 \\ -1 & 0 \end{smallmatrix}
\right )$ sends $\eta$ to $\zeta = a'(E_{2,1} - E_{1,2}) + 
a''\mathbf{i}E_{1,1}$\,.  The tangential component of $\zeta$ is
$a'(E_{2,1} - E_{1,2})$, and it has square length $2b'a'^2$.  Thus
$a'' = 0$, that is, $\eta = \eta' = a'(E_{1,2} - E_{2,1})$.

Conjugation by $g_1 = \exp(\frac{t}{2}\mathbf{i}(E_{1,2} + E_{2,1}))$
sends $E_{1,2} - E_{2,1}$ to $\zeta_1 = \cos(t)(E_{1,2} - E_{2,1}) + 
\sin(t)\mathbf{i}(E_{1,1} - E_{2,2})$, which has tangential component
$a'[\cos(t)(E_{1,2} - E_{2,1}) - \sin(t)\mathbf{i}E_{2,2}]$.
It has square length $2b'a'^2 = c^2 = a'^2[2b'\cos^2(t) + b_1\sin^2(t)]$, 
so $a' \ne 0$ implies $2b' = b_1$\,.
Similarly using $g_2 = \exp(\frac{t}{2}\mathbf{j}(E_{1,2} + E_{2,1}))$
and $g_3 = \exp(\frac{t}{2}\mathbf{k}(E_{1,2} + E_{2,1}))$ we have
$2b' = b_2 \text{ and } 2b' = b_3$\,.  Thus $b_1 = b_2 = b_3$\,, and so
$L = Sp(1)$.

We have just proved that (\ref{new-setup}) holds for 
$Sp(2)Sp(1) = S^7 \to P^1(\H) = Sp(2)/[Sp(1)\times Sp(1)]$.
The arguments of \cite[Lemmas 5.2 and 5.3]{W2020} go through to
show $\gamma \in L = Sp (1)$.  In particular $\gamma$ centralizes
$G = Sp(2)$.
\end{proof}

The immediate consequence, extending Theorem \ref{conj4sp} to include
the case $m=1$, is

\begin{theorem}\label{conj4sp2}
Let $ds^2$ be an $Sp(2)$--invariant Riemannian metric on $S^7$.
Let $\Gamma$ be a finite group of isometries of constant
displacement on $(S^7,ds^2)$.  Then the centralizer of $\Gamma$ in
$\mathbf{I}(S^7,ds^2)$ is transitive on $S^7$.  In other words
the Homogeneity Conjecture is valid for $(S^7,ds^2)$.
\end{theorem}

\medskip
\section{{\Large {\bf Summary.}}}
\label{sec7}
\setcounter{equation}{0}
\medskip

Summarizing the Theorems \ref{conj4s3}, \ref{conj4su}, \ref{conj4sp}
and \ref{conj4sp2}, 

\begin{theorem}\label{conj4}
Let $M = G/H$ be one of the spaces listed in Table \ref{obstinate-table}
above.  Let $ds^2$ be a $G$--invariant Riemannian metric on $M$, not
necessarily normal.  Let $\Gamma$ be a finite group of isometries of constant
displacement on $(M,ds^2)$.  Then the centralizer of $\Gamma$ in
$\mathbf{I}(M,ds^2)$ is transitive on $M$.  In other words
the Homogeneity Conjecture is valid for $(M,ds^2)$.
\end{theorem}

Combining Theorems \ref{conj-non-normal} and \ref{conj4} we have the
main result of this note:

\begin{theorem}\label{pos-curv}
Let $M = G/H$ be a connected, simply connected homogeneous space.  Suppose
that $M$ has a $G$--invariant Riemannian metric of strictly positive curvature. 
Let $ds^2$ be any $G$--invariant Riemannian metric on $M$, not necessarily the 
normal or the positively curved metric. Then the Homogeneity 
Conjecture is valid for $(M,ds^2)$.
\end{theorem}

\end{document}